\documentclass[11pt]{amsart}

\numberwithin{equation}{section}
\usepackage{mathrsfs,txfonts}
\usepackage[dvips]{graphicx}
\usepackage{latexsym,bm}
\usepackage{amsfonts}
\usepackage{indentfirst}
\usepackage{amsthm}
\usepackage{cases}
\usepackage{pifont}
\usepackage{amsmath}
\usepackage{amssymb}
\usepackage{enumerate}

\allowdisplaybreaks
\newtheorem{lem}{\quad\textbf{\Large Lemma}}[section]
\newtheorem{thm}[lem]{\quad\textbf{\Large Theorem}}
\newtheorem{cor}[lem]{\quad\textbf{\Large Corollary}}
\newtheorem{prop}[lem]{\quad\textbf{\Large Proposition}}
\newtheorem{de}[lem]{Definition}
\def\squarebox#1{\hbox to #1{\hfill\vbox to #1{\vfill}}}

\begin{document}
\title[On a circular formula]
{A note on a generalized  circular summation formula of theta
functions}
\author{Jun-Ming Zhu  }
\address{Department of Mathematics, Luoyang Normal University,
Luoyang City, Henan Province 471022, China}
\address{Department of Mathematics, East China Normal University,
Shanghai City 200241, China}

\email{junming\_zhu@163.com}

\thanks{The author is supported by the
 Natural Science Foundation of China
 (Grant No. 11171107 ) 
and the Foundation of Fundamental and Advanced Research of Henan Province (Grant No. 112300410024 ).
}


\begin{abstract}
In this note, we make a correction of the imaginary transformation
formula of Chan and Liu's circular formula of theta functions. We
also get the imaginary transformation formulaes for a type of
generalized cubic theta functions.

\end{abstract}
 \maketitle

\section{Introduction}
\setcounter{equation}{0}

Throughout  we put $q=e^{2\pi i\tau}$, where $\mbox{Im}\ \tau>0$. As
usual, the Jacobi theta functions $\theta_k(z|\tau)$ for $k\in\{1,
2,3,4\}$ are defined as:
\begin{eqnarray*}     
\theta_1(z|\tau) &=&-iq^{1\over8}\sum\limits_{n=-\infty}^{\infty}
                           (-1)^nq^{n(n+1)\over2}e^{(2n+1)iz}
                  ,\\
\theta_2(z|\tau)&=&q^{1\over8}\sum\limits_{n=-\infty}^{\infty}
                      q^{n(n+1)\over2}e^{(2n+1)iz}
                 ,\\
\theta_3(z|\tau)&=&\sum\limits_{n=-\infty}^{\infty}
                      q^{n^2\over2}e^{2niz}
                      , \\
\theta_4(z|\tau)&=&\sum\limits_{n=-\infty}^{\infty}
                      (-1)^nq^{n^2\over2}e^{2niz}.
\end{eqnarray*}

Circular summation of theta functions has been an interesting topic.
 In their paper \cite{chsl}, S.H. Chan
and Z.-G. Liu proved the following remarkable circular summation  formula of
theta functions \cite[Theorem 4]{chsl}.
\begin{thm} \label{chanliu}
Suppose $y_1,y_2,\cdots,y_n$ are $n$ complex numbers such that $y_1+
y_2+ \cdots+ y_n=0$. Then we have
\begin{equation}\label{chli}
\sum_{k=0}^{mn-1}\prod_{j=1}^{n} \theta_3\left(z+y_j+{k\pi\over
mn}\big|\tau\right)=
G_{m,n}(y_1,y_2,\cdots,y_n|\tau)\theta_3(mnz|m^2n\tau),
\end{equation}
where
\begin{equation}\label{gmn}
G_{m,n}(y_1,y_2,\cdots,y_n|\tau)=mn\sum_{\substack{r_1,\cdots,r_n=-\infty\\r_1+\cdots+r_n=0}}^\infty
q^{{1\over2}(r_1^2+\cdots+r_n^2)}e^{2i(r_1y_1+\cdots+r_ny_n)}.
\end{equation}
\end{thm}

This theorem generalizes the fundamental result in Zeng's \cite{ze}.
Zeng \cite{ze} is motivated  by Ramanujan's circular summation
formula \cite[page 54]{ram} (see also \cite{chln}) and Boon et al.
\cite[Eq. (7)]{bgzz}. For detailed account of the topic of the
circular summation of theta functions, we refer the readers to
\cite{chln} and \cite{chsl}.

Applying the Jacobi imaginary transformation to \eqref{chli}, S.H.
Chan and Z.-G. Liu  got the dual form of Theorem \ref{chanliu}, i.e.
\cite[Theorem 5]{chsl}.
\begin{thm} \label{chanliui},
Suppose $y_1,y_2,\cdots,y_n$ are $n$ complex numbers such that $y_1+
y_2+ \cdots+ y_n=0$. Then we have
\begin{equation}\label{chlii}
\sum_{k=0}^{mn-1}q^{k^2\over2}e^{2kiz}\prod_{j=1}^{n}
\theta_3\left(mz+(y_j+km)\pi\tau\big|m^2n\tau\right)=
F_{m,n}(y_1,y_2,\cdots,y_n|\tau)\theta_3(z|\tau),
\end{equation}
where
\begin{equation*}
F_{m,n}(y_1,y_2,\cdots,y_n|\tau)=\frac{(-i\tau)^{1-n\over2}}{(m^2n)^{n\over2}}q^{-\frac{y_1^2+y_2^2+\cdots+y_n^2}{2m^2n}}
G_{m,n}\left({y_1\pi\over m^2n},\cdots,{y_n\pi\over
m^2n}\Big|-{1\over m^2n\tau}\right).
\end{equation*}
\end{thm}

There is also another
formula (\cite[Eq. (1.10)]{chsl}) for
$F_{m,n}(y_1,y_2,\cdots,y_n|\tau)$ in \cite[Theorem 5]{chsl} without proof. Substituting
\cite[Eq. (1.10)]{chsl} back into \eqref{chlii} (i.e.   \cite[Eq.
(1.8)]{chsl} ), and then, taking $m=2$ and $n=1$, we obtain
\begin{equation}\label{fan}
\theta_3(2z|4\tau)+\theta_2(2z|4\tau)=(1+q)\theta_3(z|\tau).
\end{equation}
But
\begin{eqnarray*}
\theta_3(z|\tau)&=&\sum\limits_{n=-\infty}^{\infty}
                      q^{n^2\over2}e^{2niz}\\
&=&\sum_{n=-\infty}^{\infty}q^{(2n)^2\over2}e^{4niz}+\sum_{n=-\infty}^{\infty}q^{(2n+1)^2\over2}e^{2(2n+1)iz}\\
&=&\theta_3(2z|4\tau)+\theta_2(2z|4\tau).
\end{eqnarray*}
Thus \eqref{fan} does not hold unless $q=0$, so we can conclude that
the formula \cite[Eq. (1.10)]{chsl} is incorrect.

 In Section 2, we will  make a correction of the formula
\cite[Eq. (1.10)]{chsl} with proof.

In Section 3, we will use Theorem \ref{cli} to deduce the imaginary
transformation formulaes of a type of generalized cubic theta
functions.

\section{Correction of the  formula
\cite[Eq. (1.10)]{chsl}}

We rewrite  Theorem \ref{chanliui} plus the formula for
$F_{m,n}(y_1,y_2,\cdots,y_n|\tau)$ \cite[Theorem 5]{chsl} in the
following version, where \eqref{xsf} is the correction of \cite[Eq.
(1.10)]{chsl}.
\begin{thm} \label{cli},
Suppose $y_1,y_2,\cdots,y_n$ are $n$ complex numbers such that $y_1+
y_2+ \cdots+ y_n=0$. Then we have
\begin{equation}\label{clii}
\sum_{k=0}^{mn-1}q^{k^2\over2}e^{2kiz}\prod_{j=1}^{n}
\theta_3\left(mz+y_j+km\pi\tau\big|m^2n\tau\right)=
F_{m,n}(y_1,y_2,\cdots,y_n|\tau)\theta_3(z|\tau),
\end{equation}
where
\begin{eqnarray}
\lefteqn{F_{m,n}(y_1,y_2,\cdots,y_n|\tau)}\notag\\
&=&\frac{(-i\tau)^{1-n\over2}}{(m^2n)^{n\over2}}e^{\frac{y_1^2+y_2^2+\cdots+y_n^2}{m^2n\pi\tau
i}}
   G_{m,n}\left({y_1\over m^2n\tau},{y_2\over m^2n\tau},\cdots,{y_n\over
   m^2n\tau}\Big|-{1\over m^2n\tau}\right)\label{xsg}\\
&=&\sum_{k=0}^{n-1}q^{-{m^2k^2\over2}}\sum_{\substack{r_1,r_2,\cdots,
r_n
=-\infty\\
r_1+r_2+\cdots +r_n=k}}^\infty
q^{{1\over2}m^2n(r_1^2+r_2^2+\cdots+r_n^2)}e^{-2i(r_1y_1+r_2y_2+\cdots+r_ny_n)}.
\label{xsf}
\end{eqnarray}
\end{thm}
\begin{proof}For  the proof of
\eqref{clii} and \eqref{xsg}, see \cite{chsl}. Obviously,
\eqref{clii} can also be proved directly  by the theory of Elliptic
functions like the proof of \cite[Theorem 4]{chsl}.
 We only prove
\eqref{xsf} in detail. Applying the series expansion of
$\theta_3(z|\tau)$ to
 \eqref{clii}, the left hand side of \eqref{clii}
equals
\begin{eqnarray*}
\lefteqn{\sum_{k=0}^{mn-1}q^{k^2\over2}e^{2kiz}\prod_{j=1}^{n}
\sum_{r_j=-\infty}^\infty q^{{m^2n r_{j}^2\over2}}e^{2ir_j(mz+y_j+km\pi\tau)}}\\
&=&\sum_{k=0}^{mn-1}\sum_{r_1,\cdots,r_n=-\infty}^\infty
q^{{1\over2}[k^2+m^2n(r_1^2+r_2^2+\cdots+r_n^2)]+km(r_1+\cdots+r_n)}\\&&
\cdot e^{2i(r_1y_1+r_2y_2+\cdots+r_ny_n)}e^{2imz\
(r_1+r_2+\cdots+r_n)+2kiz}\\
&=&\sum_{l=0}^{m-1}\sum_{k=0}^{n-1}\sum_{r_1,\cdots,r_n=-\infty}^\infty
q^{{1\over2}[(km+l)^2+m^2n(r_1^2+r_2^2+\cdots+r_n^2)]+(km+l)m(r_1+\cdots+r_n)}\\&&
\cdot e^{2i(r_1y_1+r_2y_2+\cdots+r_ny_n)}e^{2imz\
(r_1+r_2+\cdots+r_n+k)+2ilz}\\
&=&\sum_{k=0}^{n-1}\sum_{r_1,\cdots,r_n=-\infty}^\infty
q^{{1\over2}[k^2m^2+m^2n(r_1^2+r_2^2+\cdots+r_n^2)]+km^2(r_1+\cdots+r_n)}\\&&
\cdot e^{2i(r_1y_1+r_2y_2+\cdots+r_ny_n)}e^{2imz\
(r_1+r_2+\cdots+r_n+k)}\\
&&+\sum_{k=0}^{n-1}\sum_{r_1,\cdots,r_n=-\infty}^\infty
q^{{1\over2}[(km+1)^2+m^2n(r_1^2+r_2^2+\cdots+r_n^2)]+(km+1)m(r_1+\cdots+r_n)}\\&&
\cdot e^{2i(r_1y_1+r_2y_2+\cdots+r_ny_n)}e^{2imz\
(r_1+r_2+\cdots+r_n+k)+2iz}\\
&&+\cdots\cdots\cdots\cdots\cdots\cdots\cdots\cdots\cdots\cdots\cdots\cdots\cdots\cdots\cdots\cdots\\
&&+\sum_{k=0}^{n-1}\sum_{r_1,\cdots,r_n=-\infty}^\infty
q^{{1\over2}[(km+m-1)^2+m^2n(r_1^2+r_2^2+\cdots+r_n^2)]+(km+m-1)m(r_1+\cdots+r_n)}\\&&
\cdot e^{2i(r_1y_1+r_2y_2+\cdots+r_ny_n)}e^{2imz\
(r_1+r_2+\cdots+r_n+k)+2iz(m-1)}.
\end{eqnarray*}
Note that, in the last identity above, the term independent of $z$
is produced only from the first sum. The right hand side of
\eqref{clii} equals
\begin{equation*}
F_{m,n}(y_1,y_2,\cdots,y_n|\tau)\sum_{n=-\infty}^{\infty}q^{n^2\over2}e^{2niz}.
\end{equation*}
Then equating the terms that are independent of $z$ on both sides of
\eqref{clii}, we find that
\begin{eqnarray*}
\lefteqn{F_{m,n}(y_1,y_2,\cdots,y_n|\tau)}\\
&=&\sum_{k=0}^{n-1}\sum_{\substack{r_1,r_2,\cdots, r_n
=-\infty\\
r_1+r_2+\cdots +r_n=-k}}^\infty
q^{{1\over2}[k^2m^2+m^2n(r_1^2+r_2^2+\cdots+r_n^2)]-k^2m^2}
e^{2i(r_1y_1+r_2y_2+\cdots+r_ny_n)}\\
&=&\sum_{k=0}^{n-1}q^{-{m^2k^2\over2}}\sum_{\substack{r_1,r_2,\cdots,
r_n
=-\infty\\
r_1+r_2+\cdots +r_n=k}}^\infty
q^{{1\over2}m^2n(r_1^2+r_2^2+\cdots+r_n^2)}
e^{-2i(r_1y_1+r_2y_2+\cdots+r_ny_n)}.\\
\end{eqnarray*}
This completes the proof.
\end{proof}
\section{The imaginary transformation formulaes for a type of generalized cubic
theta functions} We rewrite \eqref{xsg} and \eqref{xsf} as the
following corollary.
\begin{cor}\label{fg}For $n$ complex numbers $y_1,y_2,\cdots,y_n$  such that $y_1+
y_2+ \cdots+ y_n=0$ and the function
$G_{m,n}(y_1,y_2,\cdots,y_n|\tau)$ defined by \eqref{gmn}, we have
\begin{eqnarray}\label{xsff}\\
\lefteqn{G_{m,n}\left({y_1\over m^2n\tau},\cdots,{y_n\over
   m^2n\tau}\Big|-{1\over m^2n\tau}\right)}\notag\\
&=&\frac{(m^2n)^{n\over2}}{(-i\tau)^{1-n\over2}}e^{\frac{(y_1^2+\cdots+y_n^2)i}{m^2n\pi\tau}}
\sum_{k=0}^{n-1}q^{-{m^2k^2\over2}}\sum_{\substack{r_1,r_2,\cdots,
r_n
=-\infty\\
r_1+r_2+\cdots +r_n=k}}^\infty
q^{{1\over2}m^2n(r_1^2+r_2^2+\cdots+r_n^2)}e^{-2i(r_1y_1+r_2y_2+\cdots+r_ny_n)}.\notag
\end{eqnarray}
\end{cor}
Setting $m=1$ and $n=3$ in Corollary \ref{fg}, and then, replacing
$\tau$ by $\tau\over3$ , we obtain
\begin{eqnarray*}\\
\lefteqn{G_{1,3}\left({y_1\over \tau},{y_2\over \tau},{-y_1-y_2\over
  \tau}\Big|-{1\over \tau}\right)}\notag\\
&=&-\sqrt{3}i\tau\ e^{\frac{2i(y_1^2+y_2^2+y_1y_2)}{\pi\tau}}
\sum_{k=0}^{2}q^{-{k^2\over6}}\sum_{\substack{r_1,r_2, r_3
=-\infty\\
r_1+r_2 +r_3=k}}^\infty
q^{{1\over2}(r_1^2+r_2^2+r_3^2)}e^{-2i[r_1y_1+r_2y_2-r_3(y_1+y_2)]}\notag\\
&=&-\sqrt{3}i\tau\
e^{\frac{2i(y_1^2+y_2^2+y_1y_2)}{\pi\tau}}\left[\sum_{r_1,r_2
=-\infty}^\infty
q^{r_1^2+r_2^2+r_1r_2}e^{-2ir_1(2y_1+y_2)-2ir_2(y_1+2y_2)}\right.\\
&&+\sum_{r_1,r_2 =-\infty}^\infty
q^{r_1^2+r_2^2+r_1r_2-r_1-r_2+{1\over3}}e^{-2ir_1(2y_1+y_2)-2ir_2(y_1+2y_2)+2i(y_1+y_2)}\\
&&+\left. \sum_{r_1,r_2 =-\infty}^\infty
q^{r_1^2+r_2^2+r_1r_2-2r_1-2r_2+{4\over3}}e^{-2ir_1(2y_1+y_2)-2ir_2(y_1+2y_2)+4i(y_1+y_2)}\right].
\end{eqnarray*}
In the square brackets of the last identity above, replace $r_1$ and
$r_2$ in the first and second sums by $-r_1$ and $-r_2$,
respectively, and in the last sum, replace $r_1$ and $r_2$ by
 $r_1+1$ and $r_2+1$, respectively. Then we
obtain
\begin{eqnarray}\label{xsfff}
\lefteqn{G_{1,3}\left({y_1\over \tau},{y_2\over \tau},{-y_1-y_2\over
  \tau}\Big|-{1\over \tau}\right)}\\
&=&-\sqrt{3}i\tau\
e^{\frac{2i(y_1^2+y_2^2+y_1y_2)}{\pi\tau}}\left[\sum_{r_1,r_2
=-\infty}^\infty
q^{r_1^2+r_2^2+r_1r_2}e^{2ir_1(2y_1+y_2)+2ir_2(y_1+2y_2)}\right.\notag\\
&&+e^{2i(y_1+y_2)}\sum_{r_1,r_2 =-\infty}^\infty
q^{r_1^2+r_2^2+r_1r_2+r_1+r_2+{1\over3}}e^{2ir_1(2y_1+y_2)+2ir_2(y_1+2y_2)}\notag\\
&&+\left. e^{-2i(y_1+y_2)}\sum_{r_1,r_2 =-\infty}^\infty
q^{r_1^2+r_2^2+r_1r_2+r_1+r_2+{1\over3}}e^{-2ir_1(2y_1+y_2)-2ir_2(y_1+2y_2)}\right].\notag
\end{eqnarray}
\begin{de}\label{cubic}
The cubic theta functions are defined as
\begin{eqnarray}
a(x,y|\tau)&:=&\sum_{m,n =-\infty}^\infty
q^{m^2+n^2+mn}e^{2im(2x+y)+2in(x+2y)},\label{axy}\\
b(x,y|\tau)&:=&\sum_{m,n =-\infty}^\infty
\omega^{m-n}q^{m^2+n^2+mn}e^{2im(2x+y)+2in(x+2y)},\label{bxy}\\
c(x,y|\tau)&:=&\sum_{m,n =-\infty}^\infty
q^{m^2+n^2+mn+m+n+{1\over3}}e^{2im(2x+y)+2in(x+2y)},\label{cxy}
\end{eqnarray}
where $\omega:=e^{2i\pi\over3}$.
\end{de}
We have followed Chan and Liu \cite[Definition 1]{chsl} in the
definition of $a(x,y|\tau)$ above. For more discussions on cubic
theta functions of three variables, see \cite{ bha, chap, yang}

Direct computation can verify that
\begin{equation*}
b(x,y|\tau)=a\left(x,y+{\pi\over3}\big|\tau\right)
\quad\mbox{and}\quad
c(x,y|\tau)=q^{1\over3}a\left(x+{\pi\tau\over3},y+{\pi\tau\over3}\big|\tau\right).
\end{equation*}
By \eqref{gmn} and \eqref{axy}, we obtain
\begin{equation*}
G_{1,3}(y_1,y_2,-y_1-y_2|\tau)=a(y_1,y_2|\tau).
\end{equation*}
Then from \eqref{xsfff} and the discussion above, we get the
following identities.
\begin{prop}For any complex numbers $x$ and $y$, we have
\begin{eqnarray*}
\lefteqn{a\left({x\over\tau},{y\over\tau}\big|-{1\over\tau}\right)}\\&=&-\sqrt{3}i\tau\
e^{\frac{2i(x^2+y^2+xy)}{\pi\tau}}\left[a(x,y|\tau)+e^{2i(x+y)}c(x,y|\tau)+
e^{-2i(x+y)}c(-x,-y|\tau)\right],\\
\lefteqn{c\left({x\over\tau},{y\over\tau}\big|-{1\over\tau}\right)}\\
&=&\mbox{\small$-\sqrt{3}i\tau\
e^{\frac{2i(x^2+y^2+xy)}{\pi\tau}-{2i(x+y)\over\tau}}\left[a(x,y|\tau)+\omega
e^{2i(x+y)}c(x,y|\tau)+\omega^{2}
e^{-2i(x+y)}c(-x,-y|\tau)\right]$}.
\end{eqnarray*}

\end{prop}

\end{document}